\documentclass{article}

\usepackage{mathptmx}       % selects Times Roman as basic font
\usepackage{helvet}         % selects Helvetica as sans-serif font
\usepackage{courier}        % selects Courier as typewriter font
\usepackage{makeidx}         % allows index generation
\usepackage{graphicx}        % standard LaTeX graphics tool
\usepackage{multicol}        % used for the two-column index
\usepackage[bottom]{footmisc}% places footnotes at page bottom

\usepackage{amsthm}
\usepackage{latexsym}
\usepackage{dsfont}
\usepackage{bbm}
\usepackage{amssymb}
\usepackage{amsmath}
\def\eqdist{\stackrel{d}{=}}

\def\cL{\mathcal{L}}
\def\cS{\mathcal{S}}

\newcommand{\N}{\mathbb{N}}

\newcommand{\R}{\mathbb{R}}

\def\Prob{\mathbb{P}}

\newcommand{\Var}{\mathrm{Var}}
\newcommand\E{\mathbb{E}}

\newcommand{\1}{\mathbbm{1}}
\newcommand{\D}{\mathfrak{D}}

\newcommand{\bT}{\mathbf{T}}

\newcommand{\eps}{\varepsilon}

\def\eps{\varepsilon}

\newcommand{\abs}[1]{\left| #1 \right|}
\renewcommand{\P}{\mathbb{P}}
\newcommand{\Erw}[1]{{\mathbb{E}} \left[ #1 \right]}

\newtheorem{theorem}{Theorem}[section]
\newtheorem{Lemma}[theorem]{Lemma}
\newtheorem{Cor}[theorem]{Corollary}
\newtheorem{Rem}[theorem]{Remark}

\newtheorem{Prop}[theorem]{Proposition}

\allowdisplaybreaks

\begin{document}

\title{Precise tail index of fixed points of the two-sided
smoothing transform}
\author{Gerold Alsmeyer, Ewa Damek, Sebastian Mentemeier}

\maketitle

\begin{abstract}
We consider real-valued random variables $R$ satisfying the distributional equation 
\begin{equation*}
R \eqdist \sum_{k=1}^{N}T_{k}R_{k} + Q,
\end{equation*}
where $ R_{1},R_{2},...$ are iid copies of $R$ and independent of $\bT=(Q, (T_k)_{k \ge 1})$. $N$ is the number of nonzero weights $T_k$ and assumed to be a.s. finite. Its properties are governed by the function 
$$ m(s) := \E \sum_{k=1}^N \abs{T_k}^s .$$
There are at most two values $\alpha < \beta $ such that $m(\alpha)=m(\beta)=1$. We consider solutions $R$ with finite moment of order $s > \alpha$. We review results about existence and uniqueness. Assuming the existence of $\beta$ and an additional mild moment condition on the $T_{k}$, our main result asserts that
$$ \lim_{t \to \infty} t^\beta \P(\abs{R} > t)\ =\ K\ >\ 0. $$
the main contribution being that $K$ is indeed positive and therefore $\beta$ the precise tail index of $|R|$, for the convergence was recently shown by Jelenkovic and Olvera-Cravioto \cite{JO2010b}.
\end{abstract}

%\abstract{Each chapter should be preceded by an abstract (10--15 lines long) that summarizes the content. The abstract will appear \textit{online} at \url{www.SpringerLink.com} and be available with unrestricted access. This allows unregistered users to read the abstract as a teaser for the complete chapter. As a general rule the abstracts will not appear in the printed version of your book unless it is the style of your particular book or that of the series to which your book belongs.\newline\indent
%Please use the 'starred' version of the new Springer \texttt{abstract} command for typesetting the text of the online abstracts (cf. source file of this chapter template \texttt{abstract}) and include them with the source files of your manuscript. Use the plain \texttt{abstract} command if the abstract is also to appear in the printed version of the book.}

\section{Introduction}
Given a sequence $\bT:=(Q,T_{k})_{k\ge 1}$ of real-valued random variables such that
\begin{equation} N:=\sum_{k\ge 1}\1_{\{T_{k}\ne 0\}} \end{equation}
is a.s.\ finite and (w.l.o.g.) $|T_{1}|\ge...\ge |T_{N}|>|T_{N+1}|=...=0$, we consider the associated two-sided \emph{smoothing transform} (homogeneous, if $Q \equiv 0$, nonhomogeneous otherwise)
\begin{equation}\label{def:smoothing transform}
\cS:\ F\ \mapsto\ \cL\left(\sum_{k=1}^{N}T_{k}R_{k} + Q\right)
\end{equation}
which maps a distribution $F$ on $\R$ to the law of $\sum_{k=1}^{N}T_{k}R_{k} + Q$, where $R_{1},R_{2},...$ are iid random variables with distribution $F$ and independent of $\bT$. If $\cS(F)=F$, then $F$ as well as any random variable $R$ with this distribution is called a fixed point of $\cS$. In terms of random variables the fixed-point property may be expressed as
\begin{equation}\label{eq:SFPE}
R\ \eqdist\ \sum_{k=1}^{N}T_{k}R_{k} +Q
\end{equation}
where $\eqdist$ means equality in distribution. \eqref{eq:SFPE} is called a \emph{stochastic fixed point equation (SFPE)}.

It is well known that properties of fixed points of $\cS$ are intimately related to the behavior of the convex function
\begin{equation}\label{def:m}
  m(s) := \E \left( \sum_{k=1}^N \abs{T_k}^s \right). 
\end{equation}
There are at most two values $0 < \alpha < \beta$ such that $m(\alpha) = m(\beta)=1$. Assuming that both values exist, we are interested in nonzero solutions $R$ to \eqref{eq:SFPE} with finite moment of order $s>\alpha$. For a statement about existence and uniqueness of solutions with finite $\alpha$-moment see Lemma \ref{lemma:E+U} below. The situation not discussed here when solutions are mixtures of $\alpha$-stable laws (and thus having infinite $\alpha$-moment) is studied in recent work by Meiners \cite{Meiners2012}.
% We will confine to fixed points with finite expectation, and will discuss their tail behavior. 
The main result of this paper is that (under natural assumptions)
\begin{equation}\label{eq:power law of R}
\lim_{t \to \infty} t^\beta \P(\abs{R} > t)\ =\ K\ >\ 0,
\end{equation}
the main contribution actually being that the constant $K$ is positive and thus $\beta$ the precise tail index of $|R|$. The convergence was recently derived by Jelenkovic and Olvera-Cravioto \cite{JO2010b} via an extension of Goldie's implicit renewal theorem \cite{Goldie1991} to the branching case $(\Prob(N>1)>0)$, see Theorem \ref{thm:implicitrnw}  below. In the homogeneous case with nonnegative weights, \eqref{eq:power law of R} was first shown by Guivarc'h \cite{Gui1990}, see also \cite{Liu2000} and references therein.

% The restriction to fixed points with finite expectation is natural in the case of the homogeneous equation, and all interesting cases for applications are covered when considering the nonhomogeneous equation. The main part is devoted to the study of the homogeneous equation, we turn to the nonhomogeneous equation at the end, and show how results extend.
Our result is obtained by extending $r\mapsto\E \abs{R}^r$ as a holomorphic function and showing that it has a singularity at $\beta$ if and only if $K >0$.  This technique was first used in \cite{BDGHU2009} in the study of solutions to multidimensional affine recursions. We are grateful to Mariusz Mirek (personal communication) for bringing up our attention to it in the context of the branching equation \eqref{eq:SFPE} considered here.

We have organized this work as follows. Section \ref{sect:preliminaries} introduces notation and basic assumptions, provides information about the chosen setup and reviews preliminary results. Our main results are stated in Section \ref{sect:results}. Proofs are given in Section \ref{sect:proof1} with some more technical calculations deferred to Section \ref{sect:bounds}.

\section{Preliminaries}\label{sect:preliminaries}

\subsubsection*{Notations and assumptions}

For $m(s)$ defined in \eqref{def:m}, note that $m(0)= \E N$ may be infinite. We put
$$ \D := \{ s \ge 0 : m(s) < \infty\},\quad s_0 := \inf \,\D\quad\text{and}\quad
s_1:= \sup\,\D. $$
If $\D$ is nonempty, then $m$ is a convex function on $\D$. Since $m$ can be seen as
the Laplace transform of an intensity measure (see \cite[(1.8)]{AM2012}), we further have the following result (with $\Re z$ denoting the real part of a complex number $z$):

\begin{Lemma}\label{Lem:mholomorph} 
Suppose $\D = \{s\ge 0: m(s) < \infty \}$ has inner points, i.e.\ $s_0<s_1$. Then the function $m$ extends holomorphically to the strip
$s_0 < \Re z < s_1$.
\end{Lemma}

Our standing assumption throughout this paper is that 
\begin{equation}
\exists\ s_0 < \alpha < \beta < s_1:\quad m(\alpha)=m(\beta)=1.\label{A} \tag{A}
\end{equation}
Then $m'(\alpha) < 1$ and $m'(\beta) >1$.

\begin{figure}
\centering
  \includegraphics[width=\textwidth]{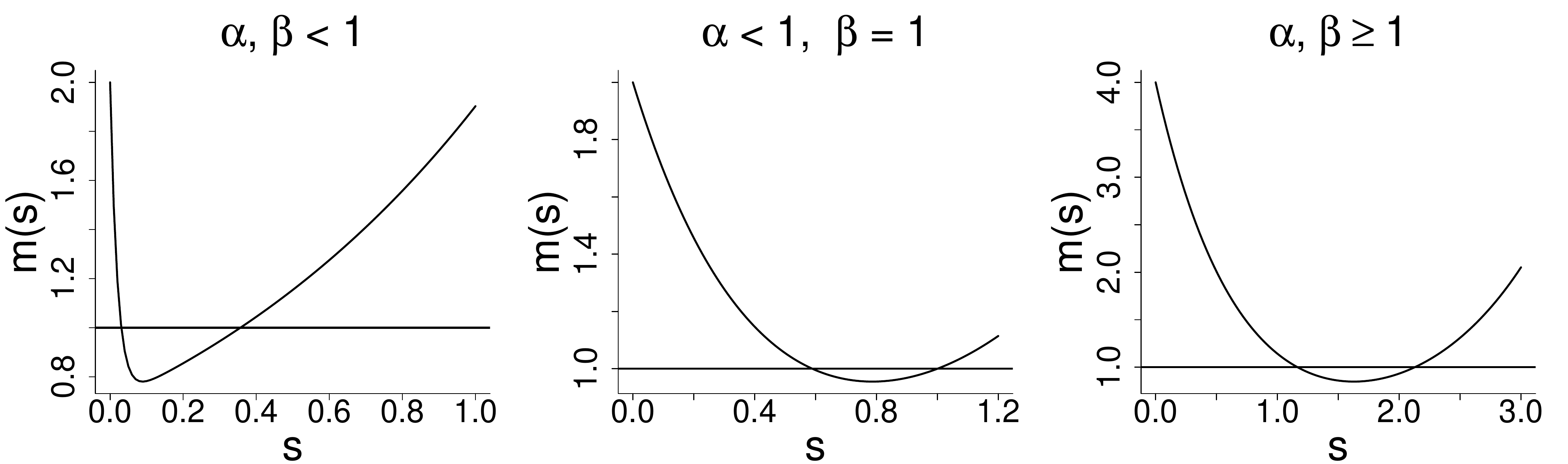} 
\caption{Three distinct cases for the values of $\alpha$ and $\beta$}
\label{fig:1}       % Give a unique label
\end{figure}

The existence of a solution $R$ to the SFPE \eqref{eq:SFPE} with finite moment of order $s>\alpha$ and a power law behavior of type \eqref{eq:power law of R} imposes some restrictions on the range of $\alpha$ to be discussed below. Our assumptions are:
\begin{itemize}
\item $1 \le \alpha <2$ if $Q=0$ a.s.\hspace{.8cm} (homogeneous case),
\item $0<\alpha<2$ if $\Prob(Q=0)<1$\quad (nonhomogeneous case).
\end{itemize}
Further conditions are needed to rule out that the tail behavior of $R$ is governed by the tails of $N$ or of $Q$. The condition on $Q$ is quite obvious,viz.
\begin{equation}
\E \abs{Q}^s < \infty \quad \text{ for all } s < s_1 . \label{B} \tag{B}
\end{equation}
Instead of moment conditions on $N$, we will impose additional conditions on the weight sums $\sum_{k=1}^{N}|T_{k}|^{s}$ and $(\sum_{k=1}^{N}|T_{k}|)^{s}$ by introducing two functions that dominate $m(s)$ for $s \ge 1$ and $s \le 1$, respectively. Define
\begin{equation}\label{def:mu} 
\mu(s) := \E \left( \sum_{k=1}^N \abs{T_k} \right)^s 
\end{equation}
and, for $\epsilon >0$,
\begin{equation}\label{def:mepsilon} 
m_\epsilon (s) := \E \left( \sum_{k=1}^N \abs{T_k}^s \right)^{1+\epsilon}.  
\end{equation}
Then $m(s) \le \mu(s)$ for $s \ge 1$, while $m(s) \le 1 + m_\epsilon(s)$ for $s \le 1$. Put
\begin{align*}
\mathfrak{D}_\mu &:= \{ s \ge 0 \ : \ \mu(s) < \infty \}, \quad s_\infty := \sup \mathfrak{D}_\mu ; \\
\mathfrak{D}_\epsilon &:= \{ s \ge 0 : m_\epsilon(s) < \infty \}, \quad s_\epsilon := \sup \mathfrak{D}_\epsilon .
\end{align*}
Our analysis will often require the study of certain moments of order $s\ge\beta$ and the distinction between the cases when $s>1,\,=1,\,<1$. Corresponding to these cases are three different sets of assumptions we introduce now, namely:
\begin{align}
&\beta < s_\infty;\label{C}\tag{C}\\
&[\beta - \delta_0, \beta] \subset \D_{\epsilon_0}\quad\text{for some }\delta_{0},\epsilon_0 >0;\label{D}\tag{D}\\
&[\beta - \delta_0,1] \subset \D_{\epsilon_0}\quad\text{for some }\delta_{0},\epsilon_0 >0.\label{D*}\tag{D*}
\end{align}
Since $\D_{\epsilon_1} \subset \D_{\epsilon_2}$ for $\epsilon_1 \ge \epsilon_2$, we may define 
\begin{equation*}
\hat{s}_\infty := 1\wedge\lim_{\epsilon \to 0} s_\epsilon
\end{equation*}
%
%Corresponding to the possible values for $\alpha$ are three distinct cases for the value of $\beta$: $\beta <1$, $\beta =1$, $\beta >1$, each calling for a slightly different assumption to guarantee that tails are not governed by $N$. On the way, we define in every case the value $s_\infty$, which will play a prominent role later.
%
%\begin{itemize}
%\item 
%  Observe that  $\mu(s) \ge m(s)$ for $s \ge 1$, so $s_\infty \le s_1$.
%\item For $\beta <1$, we have condition (C0): We assume that there are $\eps >0, \delta _0>0$ and $\beta
%<s_2 \leq \min\{1,s_1\}$ such that for every $\beta -\delta _0\leq s \leq s_2$
%\begin{equation}\label{C0} \tag{C0}
%\mu_0(s) =\E \Big (\sum _{k=1}^N |T_k|^s\Big )^{1+\eps}<\infty .
%\end{equation}
% If $s_2=1$, and condition \eqref{C}  holds, define $s_\infty$ as above. If $s_2 <1$ or \eqref{C} does not hold, set $s_\infty=s_2$.
%\item For $\beta =1$ we have condition (C1): We assume that there are $\eps >0, \delta _0>0$  such that for every $\beta -\delta _0\leq s \leq 1$
%\begin{equation}\label{C1} \tag{C1}
%\mu_0(s) =\E \Big (\sum _{k=1}^N |T_k|^s\Big )^{1+\eps}<\infty ,
%\end{equation}
%as well as condition \eqref{C} holds with some $s_\infty >1$.
%\end{itemize}
and then note that condition $\eqref{B}$ implies 
\begin{equation}\label{eq:Q-moments}
 \E \abs{Q}^s < \infty \quad \text{for all } s < \max\{s_\infty, \hat{s}_\infty\}.
\end{equation}

Finally, if $\alpha \ge 1$, we have to assume that the mean version of the SFPE \eqref{eq:SFPE} has a solution, viz.
\begin{equation}\label{mean}\tag{E}
r\ =\ r\ \E\left(\sum _{k=1}^{N}T_k\right)\ +\ \E Q 
\end{equation}
for some $r\in\R$. Note that $r$ is unique, unless $\E(\sum_{k=1}^{N}T_{k})=1$ and $\E Q=0$.

\subsubsection*{Discussion of the restrictions on $\alpha$}

The afore-stated restrictions on the range of $\alpha$, called \emph{characteristic exponent of $\bT$ or $\cS$} in \cite{ABM2012,AM2011,AM2012}, will now be justified by a number of lemmata. The first one settles the restriction $\alpha\ge 1$ in the homogeneous case.

\begin{Lemma}\label{lem:alpha<1}
Suppose that $\alpha<1$, $Q=0$ and let $R$ be a solution to \eqref{eq:SFPE} with finite moment of order $s>\alpha$. Then $R=0$ a.s.
\end{Lemma}

\begin{proof}
Plainly, we may assume $s\in (\alpha,1)$. Then \eqref{eq:SFPE} in combination with the subadditivity of $x\mapsto x^{s}$ for $x\ge 0$ provides us with
\begin{align*}
\E|R|^{s}\ &=\ \E\left|\sum_{k=1}^N T_k R_k\right|^s = \ \sum_{k=1}^\infty \E|T_k R_k|^{s}\\
&=\ \E|R|^{s}\sum_{k=1}^\infty \E|T_k|^{s} =\ \E|R|^{s}\,m(s)
\end{align*}
and thus $\E|R|^{s}=0$.
\end{proof}

\begin{Lemma}\label{lem:alpha ge 2}
Let $R$ be a nonzero solution to \eqref{eq:SFPE} with finite moment of order $s\ge 2$. Then $m(2)\le 1$ and thus $\alpha\le 2$.
\end{Lemma}

\begin{proof}
Let $\E_{\bT}$ and $\Var_{\bT}$ denote conditional expectation and conditional variance with respect to $\bT$. W.l.o.g.\ let \eqref{eq:SFPE} be valid in the stronger form $R=\sum_{k=1}^{N}T_{k}R_{k}+Q$. Then
\begin{equation}\label{eq:variance of R}
\Var\, R\ =\ \E (\Var_\bT R) + \Var(\E_\bT R).
\end{equation}
Moreover,
\begin{equation}\label{eq:cond variance of R}
\Var_\bT R\ =\ \Var_\bT\left[\sum_{k=1}^N T_k R_k + Q\right]\ =\ \sum_{k=1}^N T_k^2 \  \Var(R),
\end{equation}
whence, upon taking unconditional expectation, we obtain
\begin{align}\label{varinequality}
\infty\ >\ \Var\,R\ \ge\ \E (\Var_\bT R)\ =\ \Erw{\sum_{k=1}^N T_k^2}\Var\,R\ =\ m(2)\,\Var R\ >\ 0
\end{align}
and thus $m(2)\le 1$ as claimed.
\end{proof}

If $\alpha=2$, then \eqref{eq:cond variance of R} implies $\E\Var_{\bT}R=\Var\,R$ and thus, by \eqref{eq:variance of R}, $\Var(\E_{\bT}R)=0$. Consequently, $\E_{\bT}R$ is a.s.\ constant, in fact
\begin{equation*}%\label{eq:E_bT R relation}
\E R\ =\ \E_{\bT}R\ =\ \E R\sum_{k=1}^{N}T_{k}+Q\quad\Prob\text{-a.s.}
\end{equation*}
or, equivalently,
\begin{equation*}%\label{eq:degeneracy for Q}
Q\ =\ \left(1-\sum_{k=1}^{N}T_{k}\right)\E\,R\quad\Prob\text{-a.s.}
\end{equation*}
For the homogeneous case, we conclude that $\sum_{k=1}^{N}T_{k}=1$ a.s.\ or $\E\,R=0$ must hold. In the first subcase, the solutions to \eqref{eq:SFPE} are the normal distributions which do not have power law tails. The second subcase, which was studied by Caliebe and R\"osler \cite{CR2003}, leads to mixtures of centered normal distributions, the mixing distribution being the law of a positive constant times the nonnegative, mean one solution to the SFPE
\begin{equation}\label{eq:squared SFPE}
W\ \eqdist\ \sum_{k=1}^{N}T_{k}^{2}W_{k}.
\end{equation}
The latter solution exists and is unique if $\mu(s)<\infty$ (see also \cite[Thm.\ 2.1]{AM2012}). The following result provides the extension to the nonhomogeneous case (see also \cite[Thm.\ 2.3]{AM2012} for the case of nonnegative $T_{k}$).

\begin{Prop}\label{prop:nonhom case, alpha=2}
Suppose that $\Prob(Q\ne 0)>0$, $\alpha =2$ and that, for some $s>2$, $m(s) < 1$ and $\mu(s) < \infty$. Suppose further that
\begin{equation}\label{eq:assumption Q}
Q = r(1 -\sum_{k=1}^N T_k) \quad \P\text{-a.s.}
\end{equation}
for some $r\ne 0$. Then, for any $v\ge 0$, there is a unique solution $R$ to the SFPE \eqref{eq:SFPE} with mean $r$ and variance $v^{2}$. It is symmetric about $r$ and has characteristic function given by
\begin{equation}\label{eq:ch.f. of R, alpha=2}
\phi_R(t) \ =\ \Erw{\exp\left(irt-\frac{v^{2}t^{2}W}{2}\right)}, 
\end{equation}
where $W$ is the unique mean one solution to \eqref{eq:squared SFPE}.
\end{Prop}

\begin{proof}
If $m(s) < 1$, $\mu(s) < \infty$ for some $s>2$, then the smoothing transform is a contraction with respect to the Zolotarev metric $\zeta_{s}$ as defined in \cite{NR2004} on the subsets of probability measures with fixed first and second moment. This fact is easily derived from (a straightforward extension of) Lemma 3.1 in the afore-mentioned reference. Hence we conclude that $\cS$ has a unique fixed point with mean $r$ and arbitrary variance $v^{2}\ge 0$. Therefore, it remains to verify that $R$ with characteristic function given by \eqref{eq:ch.f. of R, alpha=2} does indeed solve our SFPE \eqref{eq:SFPE}. To this end, let $F$ be the law of $R$. Then, with $R_{1},R_{2},...$ and $W_{1},W_{2},...$ being i.i.d.\ copies of $R$ and $W$, respectively, and also independent of $\bT$, we obtain
\begin{align*}
\phi_{\cS(F)}(t)\ &=\ \Erw{\exp \left( it \sum_{k=1}^N T_k R_k + Q  \right)} \\
&=\ \Erw{\exp(itQ)\,\E_{\bT} \left( \prod_{k=1}^N \exp ( it T_k R_k ) \right)} \\
&=\ \Erw{\exp(itQ)\prod_{k=1}^N \phi_R(tT_k)}\\
&=\  \Erw{\exp\left(itr\left(1-\sum_{k=1}^{N}T_{k}\right)\right) \prod_{k=1}^{N}\E_{\bT}\left[\exp\left(irtT_{k}-\frac{v^{2}t^{2}T_{k}^{2}W_{k}}{2}\right)\right]} \\
&=\  \Erw{\exp\left(itr\left(1-\sum_{k=1}^{N}T_{k}\right)+irt\sum_{k=1}^{N}T_{k}-\frac{v^{2}t^{2}}{2}\sum_{k=1}^{N}T_{k}^{2}W_{k}\right)} \\
&=\ \Erw{\exp\left(irt-\frac{v^{2}t^{2}W}{2}\right)}\\
&=\ \phi_R(t),
\end{align*}
where we have used assumption \eqref{eq:assumption Q} on $Q$ in line four and the fixed-point property \eqref{eq:squared SFPE} for $W$ in the last line.
\end{proof}

Summarizing the situation in the case $\alpha=2$ and $\Prob(\sum_{k=1}^{N}T_{k}=1)<1$, 
the law of $R$ in \eqref{eq:ch.f. of R, alpha=2} is a $W$-mixture of normal laws with some fixed mean $r\in\R$ and variance $v^{2}w$. It exhibits a power law behavior only if this is true for (the law of) $W$ which in turn is a fixed point of the smoothing transform pertaining to $(T_{k}^{2})_{k\ge 1}$, the latter having characteristic exponent 1. With regard to \eqref{eq:power law of R}, it is therefore no loss of generality to assume $\alpha<2$ hereafter.

\subsubsection*{Existence and uniqueness of a fixed point with finite $\alpha$-moment}

The following lemma compiles results about existence and uniqueness of a solution to \eqref{eq:SFPE} with finite moment of order $\alpha$ and may be deduced from results in \cite[Section 3]{Roesler:92} and \cite[Section 3]{NR2004}. 
%For a recent reference see \cite[Theorem 5.49]{Als2012} for $\alpha \ge 1$, \cite[Theorem 5.45]{Als2012} for $\alpha < 1$, and \cite[Corollary 6.4]{Als2012}.

\begin{Lemma}\label{lemma:E+U}
Assume \eqref{A}, \eqref{B} and $\alpha <2$.
\begin{itemize}
\item[(a)] If $\alpha <1$, then there exists a unique solution $R$ to \eqref{eq:SFPE}
such that $\E \abs{R}^s < \infty$ for all $s < \beta$. It is nonzero iff $\Prob(Q \neq 0) >0$. 
\item[(b)] If $\alpha \ge 1$ and \eqref{C}, \eqref{mean} are valid, then there is a unique solution $R$ to \eqref{eq:SFPE} with $\E R=r$ (determined by \eqref{mean}) and $\E \abs{R}^s < \infty$ for all $s < \beta$. For the nonhomogeneous equation $R$ is always nonzero, and for the homogeneous one $R$ is nonzero iff $r\neq 0$.
\end{itemize}
\end{Lemma}

 \begin{Rem}\rm
Since for the homogeneous equation two nonzero solutions with distinct means are proportional, we may in fact speak of the unique nonzero solution with the property
$\E |R|^s<\infty $ for $s<\beta$ when stipulating $\E\,R =1$.
\end{Rem}

%\subsubsection*{Burkholder}
The following lemma sheds some light on the role of the function $\mu(s)$. As before, $\E_{\bT}$ denotes conditional expectation with respect to $\bT=(Q, (T_k)_{k \ge 1})$.

\begin{Lemma}\label{lemma:Burkholder}
Let $s \ge 1$ and $\mu(s) < \infty$. Then $\E
\abs{R}^s < \infty$ implies
\begin{equation} \label{eq:Burkholder}
\E_{\bT}\left( \sum_{k=1}^N \abs{T_k R_k} \right)^s\ \le\ C \left( \sum_{k=1}^N \abs{T_k} \right)^s\E \abs{R}^s  
\end{equation}
for some $C>0$ which only depends on $s$.
\end{Lemma}

\begin{proof}
This follows by an application of one of the Burkholder-Davis-Gundy inequalities (see e.g.\ \cite[Thm.\ 11.3.2]{Chow+Teicher:97}) when observing that, given $\bT$ and with $r=\E\,R$,
$$ \sum_{k=1}^{N}T_{k}(R_{k}-r)\ =\ \sum_{k\ge 1}T_{k}(R_{k}-r)\1_{\{N>n\}} $$
is the limit of the zero-mean martingale $(\sum_{k=1}^{n}T_{k}(R_{k}-r))_{n\ge 0}$, which in fact consists of finite weighted sums of i.i.d.\ random variables.
\end{proof}

Note that for $s \le 1$, we have the bound
\begin{equation}\label{eq:subadditive}
\E_{\bT} \left( \sum_{k=1}^N \abs{T_k R_k} \right)^s\ \le\ \left( \sum_{k=1}^N \abs{T_k}^s \right) \E \abs{R}^s 
\end{equation}
due to subadditivity of $x\mapsto x^{s}$ for $x\ge 0$. Taking unconditional expectation in \eqref{eq:Burkholder} and \eqref{eq:subadditive}, we arrive
\begin{align}\label{eq:bound for ER^s}
\E\left( \sum_{k=1}^N \abs{T_k R_k} \right)^s\ \le\ 
\begin{cases}
C\mu(s)\E|R|^{s},&\text{if }s\ge 1,\\
m(s)\E|R|^{s},&\text{if }s\le 1.
\end{cases}
\end{align}

\subsubsection*{The Implicit Renewal Theorem by Jelenkovic and Olvera-Cravioto}

Our analysis embarks on the following result about the tails of fixed points of two-sided smoothing transforms due to Jelenkovic and Olvera-Cravioto \cite{JO2010b}:

\begin{theorem}\label{thm:implicitrnw}
Suppose that \eqref{A} holds, that $\P({T_j}< 0)>0$ for some $j\ge 1$ and that $\P(\log \abs{T_k} \in \cdot, N\ge k)$ is nonlattice for some $k\ge 1$. Further assume \eqref{C} if $\beta>1$, and \eqref{D} if $\beta\le 1$. Let $R$ be the unique solution to \eqref{eq:SFPE}. Then
\begin{equation*}
\lim_{t \to \infty}t^\beta \P(\abs{R}>t)\ =\ \frac{K(\beta)}{m'(\beta)},
\end{equation*}
where
\begin{equation*}\label{eq:K}
K(\beta)\ :=\ \int_0^\infty \left( \P(\abs{R}>t) - \sum_{k\ge 1} \P(\abs{T_k R_k}>t) \right) t^{\beta -1}\ dt.
\end{equation*}
\end{theorem}

\begin{proof}
As will be explained at the beginning of Section \ref{sect:bounds}, conditions \eqref{C} and\eqref{D} imply the finiteness of
\begin{equation}\label{E1}
\int_0^\infty \abs{ \P(R>t) - \sum_{k\ge 1} \P(T_k R_k>t)} t^{\beta -1}\ dt 
\end{equation}
and
\begin{equation}\label{E2}
\int_0^\infty \abs{\P(R<-t) - \sum_{k\ge 1} \P(T_k R_k<-t)} t^{\beta -1}\ dt,
\end{equation}
respectively. Taking this for granted here, the stated result is \cite[Theorem 3.4]{JO2010b}.
\end{proof}

\section{Main result}\label{sect:results}

We are now ready for our main result which, loosely speaking, asserts that either $R$ has power tails of order $\beta$, or a finite moment of order $s>\beta$. 

\begin{theorem}\label{thm:main result}
Under the assumptions of Theorem \ref{thm:implicitrnw}, the following assertions hold true:
\begin{itemize}
\item[(a)] If $\beta >1$ and \eqref{A}, \eqref{B}, \eqref{C} hold true, then either $K(\beta) >0$, or $\E \abs{R}^s < \infty $ for all $s < s_\infty$.
\item[(b)] If $\beta\le 1$ and \eqref{A}, \eqref{B}, \eqref{C}, \eqref{D*} hold true, then 
either $K(\beta) >0$, or $\E \abs{R}^s < \infty $ for all $s < s_\infty$.
\item[(c)] If $\beta < 1$ and \eqref{A}, \eqref{B}, \eqref{D} hold true, then either $K(\beta) >0$, or $\E \abs{R}^s < \infty $ for all $s < \hat{s}_\infty$.
\end{itemize}
\end{theorem}

The following proposition provides a sufficient condition for $K(\beta)>0$.

\begin{Prop}\label{prop:sufficient condition}
Keeping the assumptions of Theorem \ref{thm:main result}, let $k\in\N$ be such that $\Prob(\log|T_{k}|\in\cdot,N\ge k)$ is nonlattice and assume that $\E\abs{T_k}^\gamma =1$ for some $\beta < \gamma
< s_\infty$ in parts (a), (b), resp. $\beta < \gamma < \hat{s}_\infty$ in part (c). Then
% and that $\P(\log \abs{T_j} \in \cdot, N \ge l)$
%is nonlattice, then
$$ K(\beta)>0\quad\text{iff}\quad
\P\left(r\sum_{k=1}^N T_k +Q= r\right) < 1\text{ for all }r\ne 0, $$
the latter condition being equivalent to
$$ \P\left(\sum_{k=1}^N T_k = 1\right) < 1 $$
in the homogeneous case.
\end{Prop}

Note that the existence of $\gamma$ is only a mild condition because
$$ \P\left(\max_{1 \le k \le N} \abs{T_k} > 1\right)\ =\ \P(\abs{T_1}>1)\ >\ 0. $$
Namely, if the latter failed to hold, then $m(s)$ would be decreasing function and thus $m(s)<1$ for all $s>\alpha$. But this is impossible as $m(\beta)=1$.

Note further that $\P(r\sum_{k=1}^N T_k+Q = r)< 1$ is obviously necessary for heavy tail behaviour, for otherwise $R \equiv r$ would be the unique solution with $\E R=r$.

\section{Proof of the main theorem}\label{sect:proof1}

We start with two lemmata about holomorphic functions, the first one giving a basic property of the so-called Mellin transform of a measurable function and being proved in the Appendix.

\begin{Lemma}\label{Lemma:holomorph}
Let $f : \R_\ge \to \R$ be a measurable function such that
$$ \int_0^\infty t^{s-1} \abs{f(t)} dt < \infty $$
for $s \in \{\sigma_0, \sigma_1\} \subset \R_>$. Then its Mellin transform
\begin{equation} g(z) := \int_0^\infty t^{z-1} f(t) dt \label{eq:laplace} \end{equation}
is well defined and holomorphic in the strip $\sigma_0 < \Re z < \sigma_1$.
\end{Lemma}

The next lemma, a proof of which may for instance be found in \cite[Theorem II.5b]{Widder1946}, will play a crucial role in the proof of our main result and is historically due to Landau. Its first application in the given context appears in \cite{BDGHU2009}.

\begin{Lemma}\label{Lemma:Landau}
Given the situation of Lemma \ref{Lemma:holomorph}, suppose further that $f$ is monotonic.
Let $\sigma_1:= \sup \{ s>0 \ : \ g(s) < \infty \}$ denote the {\bfseries abscissa of convergence} of $g$ . Then $g$ cannot be extended holomorphically onto any neighborhood of $\sigma_1$.
\end{Lemma}

Defining $G(s):= \E \abs{R}^s$ and $\sigma:= \sup\{s>0:G(s)<\infty\}$, we have as an immediate consequence:

\begin{Cor}\label{Cor:Landau}
The function $G$ cannot be extended holomorphically onto any neighbourhood of $\sigma$.
\end{Cor}

Put
\begin{equation}\label{eq:Ks} 
K(s)\ :=\ \int_0^\infty \left( \P(\abs{R} > t) - \sum_{k=1}^\infty \P(\abs{T_k R_k} > t) \right) t^{s-1}\ dt
\end{equation}
and suppose that \eqref{C} is valid. In order to show with the help of Lemma \ref{Lemma:holomorph} that $K$ has a holomorphic extension onto a neighborhood of $\beta$, the following proposition is crucial. Its proof will be given in Section \ref{sect:bounds}.

\begin{Prop}\label{Prop:moments}
Assuming  \eqref{A}, \eqref{B} and $\sigma\ge\beta$, it follows that
$K(\sigma +\delta) < \infty$ for some $\delta>0$ provided that, furthermore,
\begin{itemize}
\item\eqref{C} holds true and $\sigma < s_\infty$ if $\sigma>1$,
\item\eqref{C}, \eqref{D*} hold true and $\sigma < s_\infty$ if $\sigma=1$,
\item\eqref{D} holds true and $\sigma < \hat{s}_\infty$ if $\sigma<1$.
\end{itemize}
\end{Prop}

\begin{proof}[Proof of Theorem \ref{thm:main result}]
Our proof consists of two steps (tacitly assuming the respective assumptions of the theorem for the cases $\beta>,=,<1$):
\begin{itemize}
\item[\sc Step 1.] $K(\beta)=0$ iff $\sigma>\beta$.
\item[\sc Step 2.] If $K(\beta)=0$ and $\sigma>\beta$, then $\sigma=s_{\infty}$, resp.\ $=\hat{s}_{\infty}$
\end{itemize}

Before proceeding with these steps, we make the following observation (under the assumptions of the theorem):
Lemmata \ref{lemma:E+U} and \ref{lemma:Burkholder} ensure that $\E \sum_{k=1}^N \abs{T_k R_k}^s$ and $\E \abs{R}^{s}$ are both finite for $ \alpha <s <\beta$. Therefore, we may compute
\begin{align*}
K(s)\ &=\ \int_0^\infty \left( \P(\abs{R}>t) - \sum_{k=1}^\infty \P(\abs{T_k R_k}>t)  \right) t^{s-1}\ dt \notag \\
&=\ \int_0^\infty \P(\abs{R}>t) t^{s-1}\ dt\ -\ \int_0^\infty \sum_{k=1}^\infty \P(\abs{T_k R_k}>t)  t^{s-1}\ dt \notag \\
%\text{ e.g. Lemma \ref{Lemma:Goldie}} \
&=\ \frac1s \E \abs{R}^s - \frac1s \Erw{ \sum_{k=1}^N \abs{T_k R_k}^s} \notag \\
&=\ \frac1s(1-m(s))\,\E \abs{R}^s\notag
%&=\ \frac1s \left(1 - m(s)\right) \int_0^\infty \P(\abs{R}> t) t^{s-1} dt .
\end{align*}
giving
\begin{equation}\label{eq:identity} 
\frac{K(s)}{1-m(s)}\ =\ G(s)
\end{equation}
for all $s \in (\alpha, \beta)$. By Lemma \ref{Lemma:holomorph} $($with $f(t)=\P(\abs{R}>t)+\sum_{k=1}^\infty \P(\abs{T_k R_k}>t))$ and Lemma \ref{Lem:mholomorph}, both sides extend to holomorphic functions onto the strip $\alpha < \Re z < \beta$, and \eqref{eq:identity} remains valid in this strip by the identity theorem for holomorphic functions.

\vspace{.2cm}
\textsc{Step 1}. Notice that $(1-m(z))^{-1}$ has a pole of order 1 at $\beta$, for $m'(\beta)>0$, but is holomorphic otherwise in a neighborhood of $\beta$. By Proposition \ref{Prop:moments} and Lemma \ref{Lemma:holomorph}, $K(z)$ is holomorphic in the strip $\alpha<\Re z < \beta+\delta$ for some $\delta>0$. Hence, if $K(\beta)=0$, then the left-hand side (LHS) of \eqref{eq:identity} has a holomorphic extension to a neighborhood of $\beta$, and this is also an extension of the RHS, giving $\sigma>\beta$. On the other hand, $K(\beta)>0$ entails $G(\beta)=\infty$, i.e.\ $\sigma\le\beta$.

\vspace{.2cm}
\textsc{Step 2}. Now assume $K(\beta)=0$ and $\sigma>\beta$, but $ \sigma < s_\infty$, resp. $\sigma < \hat{s}_\infty$. Then, for all $\beta < \Re z < \sigma$, we have
\begin{equation*} \frac{K(z)}{1-m(z)}\ =\ G(z) ,
\end{equation*}
whence by another appeal to Proposition \ref{Prop:moments} together with Lemmata \ref{Lem:mholomorph} and \ref{Lemma:holomorph}, the LHS extends holomorphically onto $\beta < \Re z < \sigma + \delta$ for some $\delta >0$, giving an holomorphic extension of the RHS. But this is a contradiction to Corollary \ref{Cor:Landau}.
\end{proof}

We finish this section with the proof of Proposition \ref{prop:sufficient condition}.

\begin{proof}[Proof of Proposition \ref{prop:sufficient condition}]
First of all, if $K(\beta)>0$, then the uniqueness of $R$ as a solution to \eqref{eq:SFPE}
implies that $\P(r\sum_{k=1}^N T_k+Q =r)<1$ for any $r\ne 0$. In order to show the converse, suppose that $K(\beta)=0$ and thus $\E|R|^{s}<\infty$ for any $s<s_{\infty}$, resp.\ $<\hat{s}_{\infty}$. W.l.o.g. let $k=1$, so that $\E\abs{T_{1}}^{\gamma}=1$ is assumed. Putting $B:= \sum_{k=2}^N T_k R_k+Q$, the random variable $R$ satisfies the SFPE
\begin{equation}\label{eq:perpetuity} 
R\ \eqdist\ T_1 R_1 + B.
\end{equation}
Since $\E \abs{R}^\gamma < \infty$, $m(\gamma)<\infty$, and (if $\gamma >1$)
$\mu(\gamma) < \infty$, we find that the following conditions are fulfilled:
\begin{itemize}
\item[] $\E \abs{B}^\gamma < \infty$ (by Lemma \ref{lemma:Burkholder});
\item[] $\E \abs{T_1}^\gamma = 1$;
\item[] $\P(\log \abs{T_1} \in \cdot)$ is nonarithmetic;
\item[] $\E \abs{T_1}^\gamma \log^+ \abs{T_1} < \infty$.
\end{itemize}
These conditions render uniqueness of $R$ as a solution to \eqref{eq:perpetuity} and allow to invoke the results by Kesten \cite[Theorem 5]{Kesten1973} and Goldie \cite[Theorem 4.1]{Goldie1991} to infer that $\E|R|^{\gamma}<\infty$ and thus $t^{\gamma}\,\P(|R|>t)=o(1)$ as $t\to\infty$ can only hold if
$$ T_1 r + B = r\quad \text{a.s.\ for some } r \in \R  $$
or, equivalently, $R=r$ a.s. (by uniqueness) which in turn is equivalent to
$$ r\sum_{k=1}^N T_k +Q= r\quad \text{a.s.} $$
This completes our proof of the proposition.
\end{proof}

\section{Bounds for \emph{\bfseries K(s)}}\label{sect:bounds}
We proceed to a proof of Proposition \ref{Prop:moments}. This proof with $r=\beta$ and $(\cdot)^\pm$ instead of $\abs{\cdot}$ also shows the finiteness of \eqref{E1} and \eqref{E2}, thus completing the argument in the proof of Theorem \ref{thm:implicitrnw}.

% \begin{Lemma}\label{Lemma:Goldie}
% Let $X,Y$ be two nonnegative random variables, defined on a common probability space, and $s >0$. Then
% $$ \int_0^\infty \abs{\P(X>t) - \P(Y>t)} t^{s-1} dt \le \frac1s \E \abs{X^s - Y^s}, $$
% finite or infinite. If finite, absolute value signs may be removed to give
% \begin{equation}
% \label{eq:Goldie} \int_0^\infty \left( \P(X>t) - \P(Y>t) \right)t^{s-1} dt = \frac1s \E \left(X^s - Y^s\right).
% \end{equation}
% \end{Lemma}
% \begin{proof}
% see \cite[Lemma 5.3]{JO2010b}.
% \end{proof}

\begin{proof}[Proof of Proposition \ref{Prop:moments}]
By using \cite[Lemma 9.4]{Goldie1991} (in corrected form), and upon defining
\begin{align*}
H(s)\ &:=\ \E \abs{ \abs{\sum_{k=1}^N T_k R_k +Q}^s- \abs{\sum_{k=1}^N T_k R_k}^s},\\
I(s)\ &:=\ \E \abs{\abs{\sum_{k=1}^N T_k R_k }^s - \sum_{k=1}^N \abs{T_k R_k}^s},\\
J(s)\ &:=\ \Erw{ \sum_{k=1}^N \abs{T_k R_k }^s - \sup_{1 \le k \le N} \abs{T_k R_k}^s},
\end{align*}
we obtain the following estimate for $K(s)$: 
\begin{align*}
K(s)\ &=\ \int_0^\infty s t^{s-1}\left|\P\left( \abs{\sum_{k=1}^N T_k R_k+Q}>t\right) - \sum_{k=1}^\infty \P(\abs{T_k R_k}>t)\right|\ dt \\
&\le\ \int_0^\infty s t^{s-1} \abs{\P\left(\abs{\sum_{k=1}^N T_k R_k+Q}>t\right) - \P( \abs{\sum_{k=1}^N T_k R_k}>t)}\ dt \\
&\quad+\ \int_0^\infty s t^{s-1} \abs{\P\left(\abs{\sum_{k=1}^N T_k R_k}>t\right) -  \P\left(\sup_{1 \le k \le N}\abs{T_k R_k}>t\right)}\ dt \\
&\quad +\ \int_0^\infty s t^{s-1}\left(\sum_{k=1}^\infty \P(\abs{T_k R_k}>t) - \P\left(\sup_{1 \le k \le N}\abs{T_k R_k} > t\right)\right)\ dt \\
&=\ H(s)\ +\ \E \abs{ \abs{\sum_{k=1}^N T_k R_k }^s - \sup_{1 \le k \le N} \abs{T_k R_k}^s}\ +\ J(s) \\
&\le\ H(s)+I(s)+2J(s).
 \end{align*}
As for the second to last line, we note that the appearing integrand is indeed nonnegative because it is equal to $st^{s-1}\sum_{k\ge 2}\Prob(Y_{k}>t)$ where $(Y_{k})_{k\ge 1}$ denotes the decreasing order statistic of $(|T_{k}R_{k}|)_{k\ge 1}$. Then use Fubini's theorem as in \cite[Lemma 4.6]{JO2010a} to see that the pertinent integral equals $J(s)$.
The proof is completed by the next three lemmata which will show that, for some $\delta>0$, $H(s)$, $I(s)$ and $J(s)$ are bounded for all $\sigma<s<\sigma+\delta$.
\end{proof}

\begin{Lemma}\label{Lemma:qbound} Suppose that \eqref{B} holds and $\sigma\ge\beta$. If  $\sigma\ge 1$, suppose further \eqref{C} be true and $\sigma < s_\infty$. Then
$$ H(s):= \E \abs{ \abs{\sum_{k=1}^N T_k R_k +Q}^s -  \abs{\sum_{k=1}^N
T_k R_k }^s} < \infty $$ 
for all $\sigma \leq s < \sigma + \delta$ and some $\delta>0$.
\end{Lemma}

\begin{proof}%[Proof of Lemma \ref{Lemma:qbound}]
Choose $\delta\in (0,1]$ such that $\sigma+\delta<s_{\infty}$. If $s\le 1$, then (recalling \eqref{eq:Q-moments})
$$ H(s)\ \leq\ \E |Q|^s\ <\ \infty. $$ 
If $1<s<\sigma+\delta $, use the inequalities
\begin{align*}
|a^s - b^s|\ &\leq\ s(a\vee b)^{s-\delta}|a-b|^{\delta},\\
(a+b)^s\ &\leq\ 2^{s-1}(a^s+b^s),
\end{align*}
valid for $a,b\geq 0$, to infer (with $a=|\sum_{k=1}^{N}T_{k}R_{k}+Q|$ and $b=|\sum_{k=1}^{N}T_{k}R_{k}|$)
$$ H(s)\ \leq\ s(1\vee 2^{s-\delta -1})\,\Erw{|Q|^s+\abs{\sum_{k=1}^N
T_k R_k }^{s-\delta}|Q|^{\delta}}. $$ 
The last expectation is finite because, by Lemma \ref{lemma:Burkholder} and H\"older's inequality,
\begin{align*}
\Erw{\abs{\sum_{k=1}^N T_k R_k }^{s-\delta}|Q|^{\delta}}\ &\leq\ C\,\E
|R|^{s-\delta}\,\Erw{\abs{\sum_{k=1}^N |T_k|}^{s-\delta}
|Q|^{\delta}}\\
&\leq\ C\,\mu(s)^{(s-\delta)\slash s} (\E |Q|^s)^{\delta\slash s}
\end{align*}
for some constant $C\in\R_{>}$
\end{proof}

 \begin{Lemma}\label{Lemma:supbound} 
 Let $\sigma\ge\beta$. Suppose that \eqref{C} holds and $\sigma< s_\infty$ if $\sigma>1$, that \eqref{D*} holds if $\sigma=1$, and that \eqref{D} holds and $\sigma<\hat{s}_\infty$ if $\sigma<1$. Then $J(s)<\infty$ for all $0<s <\sigma+\delta$ and some $\delta>0$.
\end{Lemma}

\begin{proof}%[Proof of Lemma \ref{Lemma:supbound}]
If $\sigma\le 1$, pick $\delta\in (0,\delta _0)$ such that $\frac{\sigma+\delta}{\sigma-\delta} <1+\eps_{0}$ (and $\sigma + \delta < \hat{s}_\infty$ if $\sigma<1$).
If $\sigma>1$, pick $\delta>0$ such that $[\sigma-\delta,\sigma+\delta]\subset (1,s_\infty)$.

If $0<s<\sigma-\delta$, then $J(s)<\infty$ follows from the obvious estimate
\begin{align*}
J(s)\ \le\ \sum_{k\ge 1}\E|T_{k}|^{s}\,\E|R|^{s}\ =\ m(s)\,\E|R|^{s}.
\end{align*}
So let $s \in (\sigma-\delta,\sigma+\delta)$ hereafter. Then one can follow the proof of \cite[Lemma 4.6]{JO2010a} (replacing $(\alpha, \beta) $ and $C_i R_i$ there with $(s,\sigma-\delta)$ and $\abs{T_k R_k}$, respectively) to obtain the bound
\begin{align*}
J(s)\ &\le\ C \left(\E \abs{R}^{\sigma-\delta} \right)^{s/(\sigma -\delta)} \Erw{\left(\sum_{k=1}^N \abs{T_k}^{\sigma-\delta}  \right)^{s/(\sigma -\delta)}}\ <\ \infty\\
&=\ C \left(\E \abs{R}^{\sigma-\delta} \right)^{s/(\sigma -\delta)}m_{\eps_{0}}\left(\frac{s}{1+\eps_{0}}\right)\ <\ \infty
\end{align*}
for some constant $C\in\R_{>}$. Here we should note that, if $\sigma-\delta < 1$, the second expectation on the right-hand side is indeed finite because $s/(\sigma-\delta)<1+\epsilon_{0}$  and $\sigma-\delta<\hat{s}_{\infty}$ ensures $m_{\eps_{0}}(\sigma-\delta)<\infty$. If $\sigma-\delta \geq 1$ then we arrive at the same conclusion, for $\sum_{k=1}^N \abs{T_k}^{\sigma-\delta} \leq \big (\sum_{k=1}^N\abs{T_k}\big )^{\sigma-\delta}$.
\end{proof}

%Introducing the decreasing order statistic $(Y_{1},...,Y_{N})$ of $(|T_{1}R_{1}|,...,|T_{N}R_{N}|)$, thus $Y_{1}=\sup_{1\le k\le N}|T_{k}R_{k}|$, we obviously have
%$$ J(s)\ =\ \E\left(\sum_{k=2}^{N}Y_{k}^{s}\right). $$
%Therefore, the statement of Lemma \ref{Lemma:supbound} may be rephrased as $\sum_{k=2}^{N}Y_{k}^{s}\in L^{1}$ for all $0<s<\sigma+\delta$ and some $\delta>0$.

\begin{Lemma}\label{Lemma:sumbound} 
Let $\sigma\ge\beta$. Assume \eqref{C} and $\sigma < s_\infty$ if $\sigma >1$, \eqref{D*} if $\sigma=1$, and \eqref{D} and $\sigma < \hat{s}_\infty$ if $\sigma <1$. Then $I(s)<\infty$ for all $0<s<\sigma+\delta$ and some $\delta>0$.
\end{Lemma}

\begin{proof}%[Proof of Lemma \ref{Lemma:sumbound}]
The first part of the proof follows the argument given for \cite[Lemmata 4.8 and 4.9]{JO2010b}. Put $S:=\sum_{k=1}^{N}T_{k}R_{k}$, $S_{\pm}:=\sum_{k=1}^{N}(T_{k}R_{k})^{\pm}$ and $S_{\pm}(s):=\sum_{k=1}^{N}\big((T_{k}R_{k})^{\pm}\big)^{s}$.
Then
\begin{align*}
I(s)\ &=\ \E\left||S|^{s}-S_{+}^{s}(s)-S_{-}^{s}(s)\right|\\
&=\ \E\left|(S^{+})^{s}+(S^{-})^{s}-S_{+}(s)-S_{-}(s)\right|\\
&\le\ \E\left|(S^{+})^{s}-S_{+}(s)\right|+\E\left|(S^{-})^{s}-S_{-}(s)\right|.
\end{align*}
whence it suffices to show $\E\left|(S^{\pm})^{s}-S_{\pm}(s)\right|<\infty$ and, by an obvious reflection argument, only $\E\left|(S^{+})^{s}-S_{+}(s)\right|<\infty$. As in \cite{JO2010b}, we estimate
\begin{align}
\begin{split}\label{eq:S^+^s-S_+(s)}
\E\left|(S^{+})^{s}-S_{+}(s)\right|\ \le\ \E S_{+}(s)\1_{\{S_{+}\le S_{-}\}}\ &+\ \E\left(S_{+}^{s}-(S_{+}-S_{-})^{s}\right)\1_{\{S_{+}>S_{-}\}}\\
&+\ \E|S_{+}^{s}-S_{+}(s)|
\end{split}
\end{align}
The first two expectations on the right-hand side can be bounded by a constant times
\begin{align*}
\left( \E{\abs{R}^{s/(1+\epsilon)}} \right)^{1 + \epsilon}\E\left( \sum_{k=1}^N \abs{T_k}^{s/(1+\epsilon)} \right)^{1+\epsilon}
\end{align*}
if $\sigma<1$ (choose $a=s/(1+\epsilon)$ and $b=s\epsilon/(1+\epsilon)$ in the proof of \cite[Lemma 4.9]{JO2010b}), and by a constant times
\begin{align*}
\E|R|\,\E{\abs{R}^{s-1}}\E\left(\sum_{k=1}^{N}\abs{T_k}\right)^{s}
\end{align*}
if $\sigma\ge 1$. These bounds are finite if $0<s<\sigma+\delta$ for sufficiently small $\delta>0$ and $\epsilon<\epsilon_{0}$ with $\epsilon_{0}$ given by \eqref{D} or \eqref{D*}.

% One can follow the
%proofs of \cite[Lemma 4.8, 4.9]{JO2010b}, taking $s\in  [r, r +
%\delta)$ instead of their $\beta $. (If $r\geq 1$, $\delta $ must
%be smaller then $\frac{1}{2}$ so that $s -1 \le r - \delta$ and
%$\E \abs{R}^{s -1} < \infty$.)}{\red It works for $r=1$ too
%because then $s>1$, the condition $r-\delta >1$ is not necessary. REWRITE THIS PART!!!}
%{(If $r<1$, we need to use finiteness of $\mu (s)$, and more
%precisely, \eqref{C0} with $\eps _1<\eps $. We choose $\eps
%_1$ such that $\beta \eps _1<\delta _0$ and then $\delta <\beta
%\eps _1$.)}
It remains to show finiteness of the final expectation in \eqref{eq:S^+^s-S_+(s)}, viz.\ of
$$ L(s)\ :=\ \E \abs{ \left( \sum_{k=1}^N (T_k R_k)^+ \right)^s -
\sum_{k=1}^N \left( (T_k R_k)^+ \right)^s } $$ 
for all $0<s<\sigma+\delta$ and some $\delta>0$. We will do so by distinguishing the cases
$$ \textrm{(i)}\ \sigma<1,\quad \textrm{(ii)}\ \sigma=1,\quad\textrm{(iii)}\ 1<\sigma\le 2\quad\text{and}\quad\textrm{(iv)}\ \sigma>2. $$

(i) If $\sigma<1$, then for each $0<s\le 1$ (see also \cite[proof of Lemma 4.9]{JO2010b})
\begin{align*}
L(s)\ &=\ \Erw{\sum_{k=1}^N \left((T_k R_k)^+ \right)^{s}-\left(\sum_{k=1}^{N}(T_k R_k)^+\right)^{s}}\\
&\le\ \Erw{\sum_{k=1}^N \left((T_k R_k)^+ \right)^{s}-\max_{1\le k\le N}\left((T_k R_k)^+\right)^{s}}\\
&\le\ \Erw{\sum_{k=1}^N |T_k R_k|^{s}-\max_{1\le k\le N}\left((T_k R_k)^+\right)^{s}-\max_{1\le k\le N}\left((T_k R_k)^{-} \right)^{s}}\\
&\le\ \Erw{\sum_{k=1}^N |T_k R_k|^{s}-\max_{1\le k\le N}|T_k R_k|^{s}}\ =\ J(s),
\end{align*}
and the latter function is finite by Lemma \ref{Lemma:supbound}.
%\begin{align*}
%L(s)\ &=\ \Erw{\sum_{k=1}^N \left((T_k R_k)^+ \right)^{s}-\left(\sum_{k=1}^{N}(T_k R_k)^+\right)^{s}}\\
%&=\ \Erw{\sum_{k=1}^N |T_k R_k|^{s}-\left(\sum_{k=1}^{N}(T_k R_k)^+\right)^{s}-\sum_{k=1}^N \left((T_k R_k)^{-} \right)^{s}}\\
%&\le\ \Erw{\sum_{k=1}^N |T_k R_k|^{s}-\left(\sum_{k=1}^{N}(T_k R_k)^+\right)^{s}-\left(\sum_{k=1}^{N}(T_k R_k)^-\right)^{s}}\\
%&\le\ \Erw{\sum_{k=1}^N |T_k R_k|^{s}-\left(\sum_{k=1}^N|T_k R_k|\right)^{s}}\ \le\ J(s)
%\end{align*}
%and the latter function is finite by Lemma \ref{Lemma:supbound}.

\vspace{.1cm}
(ii) Next, let $\sigma=1$. Fix $\zeta$ such that $1 - \delta_0 < \zeta < 1$ and $(1+ \epsilon_0) \zeta >1$, where $\delta_0, \epsilon_0$ are given by condition \eqref{D*}. Then choose $\delta < \min\{(1+ \epsilon_0) \zeta -1, \zeta, 2\zeta-1 \}=(1+ \epsilon_0) \zeta -1$. Let $1 < s < 1 + \delta$ and note that $s-\zeta<1$. Applying Lemma \ref{lemma:c1functions} to $f(x)=x^s$ (thus $\xi=s-1$) and the $\zeta$ chosen above, we infer for a suitable constant $C\in\R_{>}$
\begin{align*}
L(s)\ &=\ \E \abs{\left(\sum_{k=1}^N (T_k R_k)^+ \right)^s -
\sum_{k=1}^N \left( (T_k R_k)^+ \right)^s}\\
&\leq\ C\,\Erw{\sum _{j=1}^{N-1}\left(\sum_{k=1}^j |T_k R_k|\right)^{s -\zeta}|T_{j+1} R_{j+1}|^{\zeta}}\\
&=\ C\,\Erw{ \sum _{j=1}^{N-1} \E_{\bT}\left(\left( \sum_{k=1}^j |T_k R_k|
\right)^{s -\zeta}|T_{j+1} R_{j+1}|^{\zeta}\right)}\\
&=\ C\,\E \abs{R}^{\zeta}\,\Erw{\sum _{j=1}^{N-1}  |T_{j+1}|^{\zeta}\,\E_{\bT}\left(\sum_{k=1}^j |T_k R_k|\right)^{s-\zeta}}\\
&\le\ C\,\E \abs{R}^{\zeta}\,\Erw{ \sum _{j=1}^{N-1}  |T_{j+1}|^{\zeta} \left( \E_{\bT}  (\sum_{k=1}^j |T_k R_k|)^\zeta
\right)^{(s-\zeta)/\zeta}}\\
&\le\ C\,\E \abs{R}^{\zeta}\,\Erw{ \sum _{j=1}^{N-1}  |T_{j+1}|^{\zeta} \left( \E_{\bT}  \sum_{k=1}^j |T_k R_k|^{\zeta}\right)^{(s-\zeta)/\zeta}}\\
&=\ C\,(\E \abs{R}^\zeta)^{s/\zeta}\,\Erw{ \left( \sum_{k=1}^N |T_k|^{\zeta} \right)^{s/\zeta}}\ <\ \infty
\end{align*}
where Jensen's inequality and then subadditivity have been utilized in line 5. Finiteness of the final expectation is guaranteed by \eqref{D*}.

\vspace{.1cm}
(iii) Turning to the case $1<\sigma<2$, we proceed in the same manner. Applying again Lemma \ref{lemma:c1functions} to $f(x)=x^s$ for $0<s<s_{\infty}\wedge 2$, but now with $\zeta=1$, we obtain for some $C\in\R_{>}$
\begin{align*}
L(s)\ &\leq\ C\,\E|R|\,\Erw{\sum _{j=1}^N|T_{j}|\,\E _{\bT}\left(\sum_{k=1}^{N}|T_k R_k|\right)^{s -1}}\\
&\leq\ C\,\E |R|\,\Erw{\sum _{j=1}^N|T_{j} |\left(\E_{\bT}\sum_{k=1}^{N}|T_kR_k|\right)^{s -1}}\\
& \leq\ C(\E |R|)^{s}\,\E\left(\sum_{k=1}^N |T_k|\right)^{s}\ <\ \infty
\end{align*}
where finiteness of the last expectation is guaranteed by \eqref{C}.

\vspace{.1cm}
(iv) Finally left with the case $\sigma\ge 2$, we fix again $\delta<1$ sufficiently small such that $s+\delta<s_{\infty}$. For $s\in (\sigma,\sigma+\delta)$ and small $\theta>0$, define
$$ p(\theta):= \frac{\sigma}{s-2} - \theta\quad\text{and}\quad q(\theta) := \frac{p(\theta)}{p(\theta)-1} = \frac{\sigma-\theta(s-2)}{2+2\theta-(s-\sigma) -\theta s}. $$
As one can readily check, $ \lim_{\theta\to 0} p(\theta) >1 $ and $1 < \lim_{\theta\to 0} q(\theta) <\sigma$. So we may fix $\theta >0$ so small (depending on $\delta$) that $p=p(\theta)$ and $q=q(\theta)$ for this $\theta$ satisfy
\begin{align*}
1 < p < \infty, \quad 1 < q < \sigma \quad \text{and} \quad (s-2)p < \sigma.
\end{align*}
In the following estimation, $C$ denotes a generic finite positive constant which may differ from line to line. Using Lemma \ref{lemma:c2functions} from the Appendix with $f(x)=x^{s}$, we obtain
\begin{align*}
L(s)\ &\le\ C\,\Erw{\left( \sum_{i=1}^N (T_i R_i)^+ \right)^{s-2} \sum_{1 \le j\ne k\le N} (T_j R_j)^+ (T_k R_k)^+} \\
&\le\ C\,\Erw{\left( \sum_{i=1}^N \abs{T_i R_i} \right)^{s-2} \sum_{1 \le j \neq k \le N} \abs{T_j R_j}\abs{T_k R_k}} \\
&= \ C\,\E \left( \E_{\bT} \left[ \left( \sum_{i=1}^N \abs{T_i R_i} \right)^{s-2} \sum_{1 \le j \neq k \le N} \abs{T_j R_j}\abs{T_k R_k} \right] \right) \\
&=\ C\,\E \left( \sum_{1 \le k \neq l \le N}  \E_{\bT} \left[ \left( \sum_{i=1}^N \abs{T_i R_i} \right)^{s-2}  \abs{T_k R_k}\abs{T_l R_l} \right] \right)\\
&\le\ C\,\E \left( \sum_{1 \le k \neq l \le N} \left( \E_{\bT} \left( \sum_{i=1}^N \abs{T_i R_i} \right)^{p(s-2)} \right)^{1/p}\Bigl( \E_{\bT} \abs{T_k R_k}^q \abs{T_l R_l}^q \Bigr)^{1/q} \right)\\ 
&\le\ C\,\E \left[ \sum_{1 \le k \neq l \le N} \left(\left( \sum_{i=1}^N \abs{T_i} \right)^{p(s-2)} \E\abs{R}^{p(s-2)}  \right)^{1/p}  \big(\E \abs{R}^q \big)^{2/q}\abs{T_k} \abs{T_l} \right] \\
&=C\,\big(\E\abs{R}^{p(s-2)}\big)^{1/p}\big(\E \abs{R}^q\big)^{2/q}\,\Erw{ \left( \sum_{i=1}^N \abs{T_i} \right)^{s-2} \left( \sum_{1 \le k \neq l \le N} \abs{T_k}\abs{T_l} \right) }\\
&\le\ C\,\big(\E\abs{R}^{p(s-2)}\big)^{1/p}\big(\E \abs{R}^q\big)^{2/q}\,\Erw{ \left( \sum_{i=1}^N \abs{T_i} \right)^{s-2} \left( \sum_{j=1}^N \abs{T_j} \right)^2 }  \\
&=\ C\,\big(\E\abs{R}^{p(s-2)}\big)^{1/p}\big(\E \abs{R}^q\big)^{2/q}\,\E\left(\sum_{k=1}^{N}|T_{k}|\right)^{s}\ <\ \infty
\end{align*}
where Lemma \ref{lemma:Burkholder} has been used for line 6.
\end{proof}

The previous proof gives rise to a Corollary which may be interesting in its own right:

\begin{Cor}\label{Cor:heavytail}
Let $(R_k)_{k \ge 1}$ be a sequence iid random variables independent of the random weights $(T_k)_{k\ge 1}$. Let $\sigma >1$, $0<\delta <1$ and suppose that $\E |R_1|^s <\infty $ for $s<\sigma$ and $\E (\sum _{k=1}^N |T_k|)^{\sigma+\delta}<\infty$. Then
$$ \E\left|\left(\sum _{k=1}^N(T_kR_k)^+\right)^s - \sum_{k=1}^N((T_kR_k)^+)^s\right| <\infty $$ 
for all $\sigma<s\leq \sigma+\delta$.
\end{Cor} 

\begin{proof}
If $\sigma \ge 2$ or $\sigma + \delta \le 2$, then the result is contained in the proof of Lemma \ref{Lemma:sumbound}. If $\sigma<2$, but $s:=\sigma+\delta>2$, then observe that case (iv) also works when $\sigma<2<s$. 
\end{proof}

\begin{Rem}\rm
In the case when $\sigma>1$ is not an integer, the finiteness of $L(s)$ for $0<s<\sigma+\delta$ and some $\delta>0$ sufficiently small may alternatively be inferred by the same arguments as in \cite[Proof of Lemma 5.2]{JO2010b}.
%following argument which also appears in \cite[Proof of Lemma 5.2]{JO2010b}:
%Put $\tau:=\max\{n\in\N:n\le\sigma\}$ and let $\delta>0$ be so small that $\tau<\sigma<\sigma+\delta<\tau+1$. For any $s<\sigma+\delta$ and with $\eta:=s/(\tau+1)$, we then infer
%\begin{align*}
%L(s)\ &=\ \Erw{\left(\sum_{k=1}^N\big((T_k R_k)^+\big)\right)^{(\tau+1)\eta} -
%\sum_{k=1}^N \left((T_k R_k)^+ \right)^{s}}\\
%&\le\ \E\left(\sum_{j_{1},...,j_{N}\le\tau\atop j_{1}+...+j_{N}=\tau+1}\frac{(\tau+1)!}{j_{1}!\cdot...\cdot j_{N}!}\prod_{i=1}^{N}|T_{i}R_{i}|^{j_{i}}\right)^{\eta}\\
%&\le\ \E\left(\E_{\bT}\left(\sum_{j_{1},...,j_{N}\le\tau\atop j_{1}+...+j_{N}=\tau+1}\frac{(\tau+1)!}{j_{1}!\cdot...\cdot j_{N}!}\prod_{i=1}^{N}|T_{i}R_{i}|^{j_{i}}\right)\right)^{\eta}\\
%&=\ \E\left(\sum_{j_{1},...,j_{N}\le\tau\atop j_{1}+...+j_{N}=\tau+1}\frac{(\tau+1)!}{j_{1}!\cdot...\cdot j_{N}!}\prod_{i=1}^{N}|T_{i}|^{j_{i}}\,\E|R|^{j_{i}}\right)^{\eta}\\
%&\le\ \E\left(\sum_{j_{1},...,j_{N}\le\tau\atop j_{1}+...+j_{N}=\tau+1}\frac{(\tau+1)!}{j_{1}!\cdot...\cdot j_{N}!}\prod_{i=1}^{N}|T_{i}|^{j_{i}}\,\big(\E|R|^{\tau}\big)^{j_{i}/\tau}\right)^{\eta}\\
%&=\ \|R\|_{\tau}\,\E\left(\sum_{j_{1},...,j_{N}\le\tau\atop j_{1}+...+j_{N}=\tau+1}\frac{(\tau+1)!}{j_{1}!\cdot...\cdot j_{N}!}\prod_{i=1}^{N}|T_{i}|^{j_{i}}\right)^{\eta}\\
%&\le\ \|R\|_{\tau}\,\E\left(\sum_{k=1}^{N}|T_{k}|\right)^{s}\ <\ \infty.
%\end{align*}
\end{Rem}

% maybe the following proof is not needed at all, when saying that $m$ is the Laplace transform of the intensity measure.
% \begin{proof}[Proof of Lemma \ref{Lem:mholomorph}]
% We claim that $$m(s)= s \sum_{k=1}^\infty \int_0^\infty  t^{s-1} \P(\abs{T_k}>t) dt,$$ and that the sum converges uniformely for all $1 <s < s_\infty$. Then the assertion follows by an application of the previous lemma, since $m$ then is the product of the holomorphic function $z \mapsto z$ with a uniformly convergent sum of holomorphic functions.
%
%  Let $1 \le s_0 < s_\infty$ be arbitrary (remember our basic assumption $1 \le \alpha < \beta < s_\infty$). Then for all $s \le s_0$:
% \begin{align*}
% & s \sum_{k=1}^\infty \abs{\int_0^\infty t^{s-1} \P(\abs{T_k}>t) dt} \\
% \le & s_0 \sum_{k=1}^\infty \int_0^\infty (1+t^{s_0 -1}) \P(\abs{T_k}> t) dt \\
% = & s_0 \sum_{k=1}^\infty \left[ \E \abs{T_k} + \E \abs{T_k}^{s_0}\right] \\
% = & s_0 \left[ \E \sum_{k=1}^N \abs{T_k} + \E \sum_{k=1}^N \abs{T_k}^{s_0}\right] \\
% = & s_0 \left(m(1) + m(s_0)  \right) < \infty
% \end{align*}
% The same calculation without absolute value signs gives the desired identity.
% \end{proof}

\begin{appendix}
\section*{Appendix}
\begin{proof}[Proof of Lemma \ref{Lemma:holomorph}]
We have the uniform bound
\begin{align*}
& \int_0^\infty \abs{t^{z-1} f(t)} dt = \int_0^\infty t^{\Re z -1} \abs{f(t)} dt \\
\le & \int_0^1 t^{\sigma_0 -1} \abs{f(t)} dt + \int_1^\infty t^{\sigma_1 -1} \abs{f(t)} dt < \infty .
\end{align*}
In order to show holomorphicity, take any closed path $c$ in the strip $\sigma_0 < \Re z < \sigma_1$, then we may use Fubini's theorem to infer
\begin{align*}
\int_c g(z) dz = & \int_c \left( \int_0^\infty t^{z-1} f(t) dt\right) dz \\
= & \int_0^\infty \left( \int_c t^{z-1} dz \right) f(t) dt = 0.
\end{align*}
In fact, $g$ is the Mellin-Transform of the measure $f(t) dt$.
% Substitute $t = e^u$ to see that
% $$g(z)= \int_{-\infty}^\infty e^{zu} f(e^u) du = \int_0^\infty e^{zu} f(e^u) du + \int_0^\infty e^{-zu} f(e^{-u}) du,$$
% thus $g$ is the sum of unilateral Laplace transforms of the measures $f(e^u) du$ resp. $f(e^{-u})$; and these are holomorphic functions on their domain of definition.
%
% The second integral is finite since $f$ is bounded, and for the first integral,
% \begin{align*}
% & \int_0^\infty \abs{e^{zu} f(e^u)} du = \int_0^\infty e^{u\Re z } \abs{f(e^u)} du \\
% = & \int_1^\infty t^{\Re z -1} \abs{f(t)} dt \le \int_0^\infty t^{s-1} \abs{f(t)} dt
% \end{align*}
% % Now let $\gamma$ be any closed path in the strip $0 < \Re z < s$, then (since the integrand is bounded) we may use Fubini to infer
% % $$ \int_\gamma g(z) dz = \int_{-\infty}^\infty \left(\int_\gamma e^{zu} dz\right) f(e^u) du = 0,  $$
% % thus $g$ is holomorphic.
\end{proof}

\begin{Lemma}\label{lemma:c1functions}
Let $f: \R_{\ge} \to \R_{\ge}$ be a differentiable function such that $f(0)=0$ and $f'$  is H\"older continuous of order $\xi\in (0,1]$, i.e.
$$ |f'(x_1)-f'(x_2)|\ \le\ C |x_1-x_2|^{\xi}$$
for some $C\in\R_{>}$ and all $x_{1},x_{2}\in\R_{\ge}$. Then
\begin{equation}\label{eq:c1bound} 
\left|f(s_n) - \sum_{k=1}^{n}f(x_k) \right|\ \le\ C \sum_{j =1}^{n-1} s_{j}^{1+\xi - \zeta} x_{j+1}^{\zeta}
\end{equation}
for any $\frac{1+\xi}{2}\leq \zeta \leq 1$ and $x_1,..., x_n \in \R_{\ge}$, where $s_{n}:=\sum_{j=1}^{n}x_{j}$.
\end{Lemma}

\begin{proof}
We will use induction over $n\ge 2$. For $n=2$, use $f(0)=0$ to obtain
\begin{equation}\label{eq1:c1functions}
|f(x+y)-f(x)-f(y)|\ =\ \left|\int_0^1 \left[f'(x+sy) - f'(sy) \right]y\ ds\right|\ \leq\ Cx^{\xi} y,
\end{equation}
for all $x,y\in\R_{\ge}$ which gives the result if $\zeta =1$. Otherwise, pick any $0<\sigma <1$. Then \eqref{eq1:c1functions} provides us with
\begin{align*}
|f(x+y)-f(x)-f(y)|^2\ &\le\ (Cx^{\xi}y)^{1+\sigma}\,(Cxy^{\xi})^{1-\sigma}\\
&=\ C^2x^{\xi(1+\sigma)+1-\sigma }y^{\xi(1-\sigma)+1+\sigma},
\end{align*}
which proves \eqref{eq:c1bound} for $n=2$ with $\zeta=\frac{\xi(1-\sigma)+1+\sigma}{2}$. For the inductive step $n-1\to n$, we note that
\begin{align*}
\abs{f(s_n) -\sum_{j=1}^{n}f(x_j)}\ &\le\ \abs{f(s_n) - f(s_{n-1}) - f(x_n)} + \abs{f(s_{n-1}) -\sum_{j=1}^{n-1}f(x_j)} \\
&\leq\ C\left(s_{n-1}^{1+\xi -\zeta  }x_n^{\zeta}+\sum _{j=1}^{n-2}s_{j}^{1+\xi -\zeta }x_{j+1}^{\zeta}\right)\\
&=\ C\sum _{j=1}^{n-1}s_{j}^{1+\xi -\zeta }x_{j+1}^{\zeta}
\end{align*}
which proves our claim.
\end{proof}

\begin{Lemma}\label{lemma:c2functions}
Let $f: \R_{\ge} \to \R_{\ge}$ be a twice continuously differentiable function such that $f''$ is nonnegative and increasing. Then
\begin{equation}\label{eq:c2bound} 
\left|f(s_n) - \sum_{k=1}^{n}f(x_k) \right|\ \le\ f''(s_n) \sum_{i \neq j}x_i x_j .
\end{equation}
for all $x_1,...,x_n \in \R_{\ge}$, where $s_{n}:=\sum_{j=1}^{n}x_{j}$.
\end{Lemma}

\begin{proof}
We will use induction over $n\ge 2$. For $n=2$, use $f(0)=0$ to obtain
\begin{align*}
f(x+y)-f(x)-f(y)\ &=\ \int_0^1 \left[f'(x+sy) - f'(sy) \right]y\ ds \\
&=\ \int_0^1 \left(\int_0^1 \frac{d}{dr} f'(rx+sy)\ dr \right) y\ ds \\
&=\ \int_0^1 \int_0^1 f''(rx+sy)\ x y\ dr\ ds \\
\end{align*}
By assumption $f''(rx+sy)\le f''(x+y)$ for all $r,s\in [0,1]$, whence
$$ 0 \le f(x+y)-f(x)-f(y) \le f''(x+y)\ xy $$
as asserted. For the inductive step $n-1\to n$, we note that
\begin{align*}
\abs{f(s_n) -\sum_{j=1}^{n}f(x_j)}\ &\le\ \abs{f(s_n) - f(s_{n-1}) - f(x_n)} + \abs{f(s_{n-1}) -\sum_{j=1}^{n-1}f(x_j)} \\
&\le\ f''(s_n)\,x_n\, s_{n-1} + f''(s_{n-1}) \sum_{1 \le i \neq j \le n-1}x_i x_j \\
&\le\ f''(s_n) \sum_{1 \le i \neq j \le n} x_i x_j .
\end{align*}
which proves our claim for general $n\ge 2$.
\end{proof}
\end{appendix}

\nocite{JO2010a}

\noindent \textbf{Acknowledgements}\\
G.A. and S.M. were supported by Deutsche Forschungsgemeinschaft (SFB 878). E. D. was supported by MNiSW grant
N N201 393937.

%\addcontentsline{toc}{section}{Appendix}
%
%

\noindent \textsc{Gerold Alsmeyer; Sebastian Mentemeier \\ Westf\"alische Wilhelms-Universit\"at M\"unster, \\ Institut f\"ur Mathematische Statistik, Einsteinstra\ss e 62, 48149
M\"unster. } 
gerolda@math.uni-muenster.de, mentemeier@uni-muenster.de \\

\medskip

\noindent \textsc{Ewa Damek \\ Uniwersytet
Wroclawski, \\ Instytut Matematyczny, Pl. Grunwaldzki 2/4, 50-384
Wroclaw,}\\
edamek@math.uni.wroc.pl 
\end{document}